\documentclass[a4paper, 12pt]{article}
\usepackage[sort&compress]{natbib}
\bibpunct{(}{)}{;}{a}{}{,} 

\usepackage{algorithm, algorithmicx, algpseudocode}
\usepackage{amsthm, amsmath, amssymb, mathrsfs, multirow, url, subfigure}
\usepackage{graphicx}
\usepackage{ifthen} 
\usepackage{amsfonts}
\usepackage{bbm}
\usepackage[usenames]{color}
\usepackage{fullpage}
\usepackage{csquotes}
 \DeclareMathOperator{\sign}{sign}

\theoremstyle{plain} 
\newtheorem{theorem}{Theorem}

\newtheorem{lemma}{Lemma}
\newtheorem{assumption}{Assumption}
\theoremstyle{definition}

\theoremstyle{remark}

\newtheorem{example}{Example}

\newcommand{\GG}{\mathbb{G}}

\newcommand{\PP}{\mathbb{P}}
\newcommand{\RR}{\mathbb{R}}
\newcommand{\XX}{\mathbb{X}}

\newcommand{\eps}{\varepsilon}

\newcommand{\iid}{\overset{\text{\tiny iid}}{\sim}}

\newcommand{\nm}{\mathsf{N}}

\newcommand{\dpp}{\mathsf{DP}}

\newcommand{\lap}{\mathsf{Lap}}
\newcommand{\gam}{\mathsf{Gamma}}

\newcommand{\one}{1} 
\newcommand{\diag}{\mathrm{diag}}
\newcommand{\Z}{\mathscr{Z}}

\title{Gibbs posterior inference on multivariate quantiles}
\author{
Indrabati Bhattacharya\footnote{Department of Biostatistics and Computational Biology, University of Rochester Medical Center, {\tt indrabati\_bhattacharya@urmc.rochester.edu}} \quad and \quad Ryan Martin\footnote{Department of Statistics, North Carolina State University, {\tt rgmarti3@ncsu.edu}}
}
\date{}

\begin{document}

\maketitle 


\begin{abstract}
Bayesian and other likelihood-based methods require specification of a statistical model and may not be fully satisfactory for inference on quantities, such as quantiles, that are not naturally defined as model parameters.  In this paper, we construct a direct and model-free Gibbs posterior distribution for multivariate quantiles.  Being model-free means that inferences drawn from the Gibbs posterior are not subject to model misspecification bias, and being direct means that no priors for or marginalization over nuisance parameters are required.  We show here that the Gibbs posterior enjoys a root-$n$ convergence rate and a Bernstein--von Mises property, i.e., for large $n$, the Gibbs posterior distribution can be approximated by a Gaussian.  Moreover, we present numerical results showing the validity and efficiency of credible sets derived from a suitably scaled Gibbs posterior. 

\smallskip

\emph{Keywords and phrases:} Bernstein--von Mises phenomenon; concentration rate; credible sets; learning rate; multivariate median.
\end{abstract}

\section{Introduction}
\label{S:intro}

In multivariate analysis, often the quantity of interest is the multivariate mean vector. However, there are situations when the mean is not a very good measure of location, for example, when the data is skewed, has outliers, etc. In such situations, a multivariate median would be a much more robust measure of the distribution's center.  Unfortunately, there is no universally accepted definition of a multivariate median, because there is no objective basis of ordering the data points in higher dimensions. Over the years, various definitions of multivariate medians and, more generally, multivariate quantiles have been proposed; see \citet{small1990survey} for a comprehensive review. 

The most common version of a multivariate median is called the $\ell_1$-median, which is characterized through an $\ell_1$-optimization problem.  Define 
\[ \ell_\theta(x) = \|x-\theta\|_r - \|x\|_r, \]
where $\|x\|_r = (\sum_{j=1}^d |x_j|^r)^{1/r}$ is the usual $\ell_r$-norm of a $d$-dimensional vector $x$ in $\RR^d$, for $r \in (1,\infty)$ a fixed constant.  Following \citet{small1990survey}, the $\ell_1$-median of the random vector $X \sim P$, taking values in $\RR^d$, is 
\[ \theta(P) = \arg\min_\theta P \, \ell_\theta, \]
where we use the notation $Pf = \int f \, dP$ to denote the expected value of a random variable $f(X)$ with respect to the distribution $P$.  The special case $r=2$ is called the spatial median and was studied in, e.g., \citet{brown1983statistical}.  Additional details about $\ell_1$-medians and, more generally, about multivariate quantiles are given in Section~\ref{SS:medians} below. 

In statistical applications, the distribution $P$ is unknown, but an independent and identically distributed (iid) sample $X_1,\ldots,X_n$ of random vectors in $\RR^d$ are available from $P$.  Based on the above formulation, an immediate strategy for estimating the $\ell_1$-median is to follow the definition of $\theta(P)$ but replace the distribution $P$ with the empirical distribution, i.e., with $\PP_n = n^{-1} \sum_{i=1}^n \delta_{X_i}$, where $\delta_x$ denotes a point-mass distribution at the point $x \in \RR^k$.  That is, the standard point estimate of $\theta(P)$ is 
\begin{equation}
\label{eq:sample.median}
\hat\theta_n = \theta(\PP_n) =\arg\min_{\theta} \PP_n \, \ell_\theta.
\end{equation}
The spatial sample median is a highly robust estimator of location, in particular, its breakdown point is $1/2$. Also, \citet{mottonen2010asymptotic} investigated the asymptotic properties of spatial median and proved its asymptotic normality.  

Beyond estimation, if the goal is probabilistic inference on a multivariate median or quantile, i.e., via a ``posterior distribution,'' then the chief difficulty is that these are not naturally described as parameters in a statistical model.  That is, no standard or otherwise ``reasonable'' model for multivariate data will include a quantile in its parametrization, so some potentially dangerous non-linear marginalization \citep[e.g.,][]{fraser2011, martin.nonadditive} would typically be required.  More importantly, with specification of a model comes the risk of model misspecification bias, and, since our quantity of interest is well-defined without a model, it is not clear what can be gained by working in a model-based framework to balance out the risk of misspecification bias.  
However, there are advantages to having a posterior probability distribution on which to base inferences; for example, a posterior density provides a nice visual summary of uncertainty, credible regions can be immediately read off from it without asymptotic approximations, and prior information about the quantity of interest can be incorporated whenever it is available.  A Bayesian approach that both achieves these desirable features (more or less) and avoids the risk of model misspecification bias must be nonparametric.  That is, assign a prior distribution---say, a Dirichlet process \citep{ferguson1973}---to the infinite-dimensional $P$, get the corresponding posterior, and then do the non-trivial marginalization to $\theta=\theta(P)$.  
We call this an {\em indirect} approach.  Aside from computational challenges, a downside of the indirect approach is that incorporating available prior information about the quantile is difficult; probably the best option is to choose a Dirichlet process base measure to have quantile equal to a prior guess, but it is not clear how this (and other features of the specified base measure) affect the marginal posterior for the quantile.  

Is it possible to develop a posterior for the multivariate quantile in a more {\em direct} way, without marginalization, etc.?  Here we investigate the construction of a {\em Gibbs posterior} for a multivariate quantile.  On one hand, like M-estimation, this approach uses a suitable loss function, rather than a likelihood, to connect the quantity of interest to the observed data, which eliminates the risk of model misspecification bias.  On the other hand, like Bayesian inference, it produces a genuine posterior distribution and allows for the direct incorporation of prior information.  After some background about multivariate quantiles and Gibbs posteriors in Section~\ref{S:background}, we define our object of interest, namely, the Gibbs posterior distribution for a multivariate quantile, and investigate its properties.  In particular, in Section~\ref{SS:theory}, we first establish that the Gibbs posterior concentrates around the true quantile at the usual root-$n$ rate and, second, that it has an asymptotic Gaussian approximation in the Bernstein--von Mises sense.  Unfortunately, the covariance matrix in this Gaussian approximation is ``wrong'' in the sense that it does not match that of the M-estimator around which it is centered.  Fortunately, the Gibbs posterior depends on a user-specified {\em learning rate} \citep[e.g.,][]{bissiri2016general, grunwald.ommen.scaling, syring2018calibrating} which can be tuned to at least partially correct for the covariance matrix mismatch. We use a bootstrap-based calibration algorithm proposed by \citet{syring2018calibrating} for choosing the learning rate, which we describe in Section~\ref{SS:learning}.  In Section~\ref{SS:sim}, we  compare the finite-sample performance of our proposed Gibbs posterior inference to that based on existing Bayesian approaches.  The two key take-aways are: (a)~the Gibbs posterior outperforms the model-based parametric Bayesian approach in misspecified situations, and (b)~the Gibbs posterior credible sets are at least as good as the nonparametric Bayes credible sets in terms of coverage but they are more efficient in some cases.  We also apply the Gibbs posterior approach to infer the spatial median based on a real data set in Section~\ref{SS:real}, for which the assumption of normality is debatable.  We show that the Gibbs posterior outperforms a normality-based Bayesian solution in terms of out-of-sample risk, implying that our Gibbs solution avoids the inherent bias coming from the normality assumption.  Some concluding remarks are given in Section~\ref{S:discuss} and proofs of the two main theorems are presented in the Appendix.  


\section{Background}
\label{S:background}

\subsection{Multivariate quantiles}
\label{SS:medians}

Again, the lack of a well-defined ordering of multivariate observations creates a major issue in defining multivariate quantiles. \citet{abdous1992note} and \citet{babu1989joint} investigated the coordinate-wise medians and quantiles.  However, the coordinate-wise quantiles do not provide much information about the joint distribution of the vector and they also lack some desirable geometric properties, namely, rotational invariance. 
To fill this gap, \citet{chaudhuri1996geometric} introduced the notion of geometric quantiles based on the geometric configuration of multivariate data clouds.  These quantiles are natural generalizations of the $\ell_1$-median. For univariate observations $X_1,\dots,X_n \in \RR$, the sample $\alpha$th quantile $\alpha \in (0,1)$ is obtained by minimizing $\xi \mapsto \sum_{i=1}^n \{|X_i-\xi| +u (X_i-\xi)\}$, with $u=2\alpha -1$. Generalizing this idea to higher dimensions, the $d$-dimensional geometric quantiles, with $\ell_r$-norm, are indexed by points in the open unit ball $B_q^{(d)}=\{u \in \RR^d: \| u \|_q < 1\}$, where $q$ is the H\"older conjugate of $r$, i.e., $r^{-1}+q^{-1}=1$. Thus, for $u \in B_q^{(d)}$, the $d$-dimensional sample $u$-quantile is then defined as
\begin{equation}
    \widehat{Q}_n(u)= \arg \min_{\xi \in \RR^d}\frac{1}{n} \sum_{i=1}^n \Phi_r(u, X_i-\xi),
    \label{eq:sample.quantile}
\end{equation}
where $\Phi_r(u,t)=\| t \|_r + \langle u,t \rangle$, with $\langle \cdot,\cdot \rangle$ being the usual inner product. It is easy to see that $\widehat{Q}_n(0)$ is the same as the $\ell_1$-median; $\hat{\theta}_n$. The population analog of $\widehat{Q}_n(u)$ is given by
\begin{equation}
\label{eq:quantile}
    Q_P(u)= \arg\min_{\xi \in \RR^d}P\{\Phi_r(u,X-\xi)-\Phi_r(u,X)\}.
\end{equation}
\citet{chaudhuri1996geometric} showed that the geometric quantiles are both equivariant under location transformation and homogeneous scale transformation of the individual coordinates. 

Chaudhuri's approach has received considerable attention in the literature, and has also been extended to regression contexts, for example, in \citet{chakraborty1999, chakraborty2003}.  One other notable approach to generalizing univariate quantiles to multivariate case is the directional quantile approach developed by \citet{hallin2010multivariate}. A directional quantile $\tau$ is a function of two components, namely, a direction vector $u$ and a depth $\gamma \in (0,1)$. Then the $\tau=u\gamma$ directional quantile, denoted by $\lambda_{\tau}$ is a hyperplane through $\RR^d$. 

In a Bayesian setting, \citet{bhattacharya2019bayesian} considered the use of a Dirichlet process prior on the underlying distribution $P$, and explored properties of the corresponding marginal posterior distribution of $Q_P(u)$ or, more precisely, a Bayesian bootstrap approximation thereof.  


\subsection{Gibbs posterior distributions}
\label{SS:gibbs}

The Gibbs measure has its origins in statistical physics but a version of it has received attention in the statistics, machine learning, and econometrics literature; see, e.g., \citet{bissiri2016general}, \citet{zhang2006a, zhang2006b}, and \citet{chernozhukov2003mcmc}.  Some recent statistical applications include data mining \citep{jiang2008gibbs}, clinical trials \citep{syring2017gibbs}, image analysis \citep{syring2016robust}, actuarial science \citep{syring2019gibbs}, and classifier performance assessment \citep{wang.martin.auc}.  Below we define the Gibbs posterior and some features that will be relevant in what follows.  

Let $X^n = (X_1,\ldots,X_n)$ be an iid sample from some distribution $P$.  Suppose there is some functional $\theta=\theta(P)$ that we are interested in estimating and making inference about.  By the way this problem has been stated, it should be clear that $\theta$ generally cannot be understood as a model parameter, so we cannot expect that there is a likelihood function that can be used to connect the data to the quantity of interest.  Instead, the setup assumes that the functional is defined via an optimization problem.  That is, there exists a function $\ell_\theta(x)$ such that the true value $\theta^\star$ of $\theta(P)$ is the minimizer of the function $R(\theta) = P\ell_\theta$; here, note that, as is customary in the literature, we denote the quantity of interest and a generic value of it with the same symbol, $\theta$, and distinguish the true value $\theta^\star$ where necessary.  The function $\ell_\theta$ is called the {\em loss} and $R(\theta)$ the corresponding {\em risk}.  Since we do not know $P$, we also do not know the risk, so inference on $\theta$ requires that we replace $P$ with the observed data in some way.  In particular, define the empirical risk as $R_n(\theta) = \PP_n \ell_\theta$, where $\PP_n$ is the empirical distribution of the data $X^n$.  The estimator $\hat\theta_n$ derived by minimizing $R_n(\theta)$ is often called an {\em M-estimator} \citep[e.g.,][]{huber1981}.  

Empirical risk minimization is a common task in machine learning and can be challenging because the data-dependent objective function $R_n$ is not always well-behaved.  As an alternative to optimization, the {\em PAC-Bayes} literature \citep[e.g.,][]{mcallester1999, alquier2008}---where PAC stands for probability approximately correct---proposed to construct a distribution that concentrates on $\theta$ values for which $R_n(\theta)$ is small.  That distribution is the Gibbs posterior and is given by 
\begin{equation}
\Pi_n(B)=\frac{\int_B e^{-\omega nR_n(\theta)} \, \Pi(d\theta)} {\int_{\RR^d} e^{-\omega nR_n(\theta)} \, \Pi(d\theta)}, \quad B \subseteq \RR^d,
\label{eq:gibbs}
\end{equation}
where $\Pi$ is a prior distribution and $\omega > 0$ is called the {\em learning rate}. The choice to use the empirical risk function in \eqref{eq:gibbs} is crucial.  Here we adopt the perspective of \citet{bissiri2016general}, who show that, when---like in our present application---the quantity of interest is defined as the minimizer of the expected loss, $R(\theta) = P \ell_\theta$, the {\em proper generalization of Bayesian inference}, in the sense of coherent updates of beliefs, leads to the formulation in \eqref{eq:gibbs} with the empirical version of the expected loss, $R_n(\theta) = \PP_n \ell_\theta$.  So, while other choices of ``pseudo-posterior distributions'' are possible, these lack justification as a proper generalization of Bayesian inference. 

It should be emphasized that the introduction of the learning rate $\omega > 0$ in \eqref{eq:gibbs} is not an arbitrary choice being made by us, it is a technical artifact of the generalized Bayes posterior construction in, e.g., \citet{bissiri2016general}.  Both the prior and learning rate play crucial roles in determining the Gibbs posterior's practical performance.  Unfortunately, neither are fully determined by the context/data, so effort is required from the user.  
\begin{itemize}
\item Of course, one can use a vague/flat prior, which we do in our simulation study below to avoid confounding the Gibbs posterior's performance with the effects of using an informative prior, but this is not our recommendation.  Since a multivariate quantile is a real-world quantity, and not the parameter of an artificial model, it is possible that genuine prior information is available and, in such cases, that prior information absolutely should be used.  
\vspace{-2mm}
\item The learning rate is analogous to the tuning parameter in machine learning algorithms in the sense that there is no ``true'' $\omega$ that can be learned from the data.  However, data-driven choices of the learning rate are still possible, and a number of such methods have been proposed in the recent literature, e.g., \citet{grunwald2012}, \citet{holmes.walker.scaling}, \citet{lyddon.holmes.walker}, and \citet{syring2018calibrating}; a comparison of these can be found in \citet{gpc.compare}.  More details about learning rate selection are given in Section~\ref{SS:learning}.  
\end{itemize}

\section{Gibbs posteriors for multivariate quantiles}
\label{S:gibbs}

\subsection{Definition}
\label{SS:def}

Suppose we have an iid sample $X_1,\dots,X_n$ from a distribution $P$ on $\RR^d$. Since the $\ell_1$-median $\theta(P)$ is the same as $Q_P(0)$, we will discuss the Gibbs posterior construction for a geometric quantile $Q(u)$ with $\ell_r$-norm for some fixed $r \in (1,\infty)$ and fixed $u \in B_q^{(d)}$. For simplicity, we will denote $\hat{\theta}_n=\hat{Q}_n(u)$ and $\theta^{\star}=Q_P(u)$ from now on. 

Since the quantity of interest is the minimizer of a function defined by an expectation, in \eqref{eq:quantile}, it makes sense to define the loss $\ell_\theta$ as that function inside the expectation.  Precisely, we take $\ell_\theta(x) = \Phi_r(u,x-\theta)$, where, again, $u$ and $r$ are fixed.  Then the risk is $R(\theta) = P \ell_\theta$, minimized at $\theta^\star$, and the empirical risk is 
\begin{equation}
\label{eq:erisk}
    R_n(\theta)=\PP_n \ell_\theta = \frac{1}{n}\sum_{i=1}^n\{\| X_i-\theta \|_r+\langle u, X_i-\theta \rangle\}, 
\end{equation}
minimized at $\hat\theta_n$.  Given a prior distribution $\Pi$ for $\theta$, the Gibbs posterior distribution $\Pi_n$ is defined like in \eqref{eq:gibbs}.  It follows from H\"older's inequality that $R_n(\theta) \geq 0$, so if the prior is proper, then the denominator in \eqref{eq:gibbs} is finite and the Gibbs posterior is well-defined.  Therefore, $\Pi_n$ is just an ordinary probability distribution, and features of that distribution can be extracted and summarized via the standard Markov chain Monte Carlo methods.  For the theoretical analysis that follows, we assume that the learning rate $\omega$ is a fixed constant, but we will recommend a data-driving choice of $\omega$ in Section~\ref{SS:learning}.  

While the definition of the Gibbs posterior in \eqref{eq:gibbs} appears to be only a modest generalization of the familiar Bayesian definition, there are some important differences that deserve emphasis. First, in this multivariate quantile setting, $\theta$ is not a parameter indexing a statistical model, so {\em there is no likelihood function for $\theta$} and, consequently, no direct Bayesian posterior distribution for $\theta$.  Therefore, any ordinary Bayesian approach to this problem would necessarily be indirect, i.e., define a statistical model with parameter, say, $\psi$, introduce a prior distribution for $\psi$, evaluate the corresponding posterior, and then marginalize to $\theta$, which is now a function of $\psi$.  So, despite the superficial similarity between the Gibbs and Bayesian posterior distributions, they are in fact quite different.  Second, along similar lines, since $\theta$ is a real-world quantity that exists independently of a statistical model, there might be genuine prior information available about it.  That the Gibbs posterior is direct implies that this prior information can be readily incorporated in \eqref{eq:gibbs}.  Compare this to an indirect Bayesian posterior distribution where the model parameter $\psi$ often has limited real-world interpretation and, therefore, requires a non-informative prior, making it impossible to incorporate whatever prior information about $\theta$ might be available in a given application.


\subsection{Asymptotic properties}
\label{SS:theory}

First, we investigate the Gibbs posterior concentration rate, i.e., the radius of the smallest ball around $\theta^\star$ to which the posterior asymptotically assigns all of its mass, as $n \to \infty$.  For this, we require a mild condition on the underlying distribution $P$.  

\begin{assumption}
\label{asmp:density1}
$P$ admits a density $p$ that is continuous and bounded away from 0 on a compact set $\XX \subset \RR^d$ containing $\theta^{\star}$ and having non-empty interior.
\end{assumption}

\begin{assumption}
\label{asmp:density2}
The density $p$ is bounded away from $\infty$ on compact subsets of $\RR^d$.  
\end{assumption}

An important consequence (see the proof of Lemma~\ref{lem:separation} in the Appendix) of Assumptions~\ref{asmp:density1}--\ref{asmp:density2} is that the function $R$ is twice differentiable at $\theta^\star$, where $\dot R(\theta^\star) = 0$ and $V_{\theta^\star} := \ddot R(\theta^\star)$ is positive definite; here, dot and double-dot correspond to first and second derivatives with respect to $\theta$, the gradient vector and the Hessian matrix, respectively.  

\begin{assumption}
\label{asmp:prior}
The prior distribution $\Pi$ has a density $\pi$ which is continuous and bounded away from 0 in a neighborhood of $\theta^{\star}$.
\end{assumption}

\begin{theorem}
\label{thm:rate}
Under Assumptions~\ref{asmp:density1}--\ref{asmp:prior}, for any $\omega > 0$, the Gibbs posterior $\Pi_n$ for the multivariate quantile satisfies 
\[ P^n \Pi_n(\{\theta \in \RR^d: \|\theta - \theta^\star\|_2 > a_n n^{-1/2}\}) = o(1), \quad n \to \infty, \]
where $a_n \to \infty$ is any diverging sequence.
\end{theorem}

\begin{proof}
See Appendix~\ref{proof:rate}.
\end{proof}

The concentration rate result in Theorem~\ref{thm:rate} holds for all $\omega > 0$, that is, there are no restrictions on $\omega$ needed to achieve the target root-$n$ rate.  However, the limiting Gibbs posterior distribution shape in Theorem~\ref{thm:bvm} below and, hence, the practical performance of the Gibbs posterior distribution, does depend on $\omega$, so we recommend a data-driven choice as described in Section~\ref{SS:learning}.


Next, we will prove a Bernstein--von Mises theorem for the Gibbs posterior, that is, the Gibbs posterior can be approximated by a Gaussian distribution in a total variation sense as $n \to \infty$.  
Before formally stating this result, we need a bit more notation.  The loss $\theta \mapsto \ell_\theta(x)$ can be differentiated for $P$-almost all $x$, and the $j^\text{th}$ component of the gradient vector, $\dot\ell_\theta(x)$, is given by
\[ \dot\ell_\theta(x)_j = \frac{\vert x_j-\theta_j \vert ^{r-1}}{\|x-\theta\|_r^{r-1}}\mathrm{sign}(\theta_j-x_j) - u_j, \quad j=1,\dots,d, \]
where, again, $r$ and $u$ are fixed, and $\mathrm{sign}(\cdot)$ denotes the {\em signum} function.  Now set $\Delta_{n,\theta^\star} = n^{-1/2} \sum_{i=1}^n V_{\theta^\star}^{-1} \dot\ell_{\theta^\star}(X_i)$.  

\begin{theorem}
\label{thm:bvm}
Under Assumptions~\ref{asmp:density1}--\ref{asmp:prior}, the sequence of centered and scaled Gibbs posteriors, with any learning rate $\omega > 0$, approaches a sequence of $d$-variate normal distributions in total variation, that is,  
\[ \sup_B\bigl| \Pi_n(\{\theta: n^{1/2}(\theta-\theta^{\star}) \in B\})-\nm_d(B \mid \omega\Delta_{n,\theta^{\star}},{(\omega V_{\theta^{\star}})}^{-1})\bigr| = o_P(1), \quad n \to \infty. \]
\end{theorem}

\begin{proof}
See Appendix~\ref{proof:bvm}.
\end{proof}

A few technical remarks about the theorem and its proof are in order.
\begin{itemize}
\item The proof proceeds by checking the available sufficient conditions for Bernstein--von Mises theorems, e.g., in \citet{chernozhukov2003mcmc}, \citet{kleijn2012bernstein}, etc.  Here we opt to follow the latter reference whose results are more flexible and easier to apply in other similar applications.
\vspace{-2mm}
\item A similar result would hold in examples other than multivariate quantiles.  The critical condition is that the empirical risk $R_n$ satisfies a version of the {\em local asymptotic normality} condition, i.e., for every compact set $K \subset \RR^d$,
\begin{equation}
\label{eq:lan}
\sup_{h \in K} \Bigl| n\{R_n(\theta^{\star}+h n^{-1/2}) - R_n(\theta^{\star})\}-  h^\top V_{\theta^{\star}} \Delta_{n,\theta^{\star}}-\tfrac{1}{2}h^\top V_{\theta^{\star}}h \Bigr| = o_P(1).
\end{equation}
So we can expect similar conclusions in any other problem for which \eqref{eq:lan} holds.  
\vspace{-2mm}
\item Finally, like in \citet[][p.~144]{van2000asymptotic}, the normal approximation can be centered about an estimator that is asymptotically equivalent to $\Delta_{n,\theta^\star}$.  In particular, using the location shift invariance of the total variation distance, it follows from Theorem~\ref{thm:bvm} that the Gibbs posterior $\Pi_n$ is approximately $\nm_d(\hat\theta_n, (\omega n V_{\theta^\star})^{-1})$.  
\end{itemize}

In a Bayesian setting, with a regular, well-specified model, a Bernstein--von Mises theorem ensures that inferences derived from the Bayesian posterior distribution are valid in a frequentist sense.  For example, a $100(1-\alpha)$\% posterior credible set will have frequentist coverage probability approximately equal to $1-\alpha$ for large $n$.  The reason for this Bayesian--frequentist connection is that the posterior distribution centers around, in that case, the maximum likelihood estimator, and the covariance matrix in the Bernstein--von Mises theorem is the inverse Fisher information matrix, which agrees with the asymptotic covariance matrix of the maximum likelihood estimator.  However, when the model is misspecified, like in \citet{kleijn2012bernstein}, or, like here, where no model is specified at all, then this covariance matching is not guaranteed.  Indeed, in our present case, the covariance matrix in the normal approximation to the Gibbs posterior is ${(\omega V_{\theta^\star})}^{-1}$ whereas the asymptotic covariance matrix of $n^{1/2}(\hat\theta_n-\theta^\star)$ is 
\begin{equation}
\label{eq:Gamma}
\Gamma = V_{\theta^{\star}}^{-1}P(\dot{\ell}_{\theta^{\star}}\dot{\ell}_{\theta^{\star}}^\top)V_{\theta^{\star}}^{-1}, 
\end{equation}
which comes from the familiar sandwich formula.  Since these two matrices are generally different, our Bernstein--von Mises theorem does not guarantee that inference drawn from the Gibbs posterior are valid in a frequentist sense.  One way to avoid this covariance mismatch is 
to replace the empirical risk $R_n$ in \eqref{eq:gibbs} with, e.g., a quadratic form like
\[ \theta \mapsto \dot R_n(\theta)^\top \{\PP_n(\dot{\ell}_{\theta}\dot{\ell}_{\theta}^\top)\}^{-1} \dot R_n(\theta), \]
with the covariance mismatch correction term squeezed in.  Another, following a suggestion in \citet{yang2016posterior}, is to directly define a data-dependent distribution 
\begin{equation}
\label{eq:dumb.normal}
\widetilde \Pi_n(A) = \nm_d(A \mid \hat\theta_n, n^{-1} \widehat\Gamma), 
\end{equation}
where $\widehat\Gamma$ is a suitable estimator of $\Gamma$ as in \eqref{eq:Gamma}.  Both of these ``pseudo-posterior distributions'' have their merits, but they lack the interpretation of being proper generalized Bayes posteriors, which is our focus in this paper.

If we were in a traditional Bayesian setting and were unfortunate enough that our model was sufficiently misspecified that we get the aforementioned covariance matrix mismatch, then (a)~we typically would not be aware of this problem and (b)~there would be nothing we could do about it, aside from starting over with a different model.  However, since we are working within a Gibbs framework, we are aware of and openly acknowledge that our posterior distribution is based on an effectively misspecified model and, moreover, we have a potential remedy: adjusting the learning rate. 

It is easy to see that if $P(\dot{\ell}_{\theta^{\star}}\dot{\ell}_{\theta^{\star}}^\top) \propto V_{\theta^{\star}}$, i.e., if the {\em generalized information equality} \citep{chernozhukov2003mcmc} holds, then the covariance matrix in the normal approximation to the Gibbs posterior will be proportional to the asymptotic covariance matrix of the M-estimator $\hat{\theta}_n$.  In that case, we can exactly correct for the covariance mismatch simply by tuning the learning rate.  In general, however, simply tuning the scalar learning rate parameter cannot fully correct for the covariance mismatch, but it is still possible to find a learning rate such that credible sets derived from the Gibbs posterior have approximately the nominal frequentist coverage probability; see Section~\ref{SS:learning}.

\subsection{Choice of the learning rate}
\label{SS:learning}

As we indicated in Section~\ref{S:background}, the choice of learning rate is critical to the performance of methods derived from a Gibbs posterior distribution.  This is especially important in our present situation because, as mentioned in the remarks following Theorem~\ref{thm:bvm}, the Gibbs posterior does not inherit the correct asymptotic shape.  This covariance mismatch is a common occurrence when a Bayesian model is misspecified but, unlike the traditional Bayesian setting where nothing can be done to overcome the misspecification bias, the Gibbs posterior has a learning rate that can be suitably chosen to correct for the mismatched asymptotic covariance matrix.  

More specifically, following \citet{syring2018calibrating}, if $\Pi_n^\omega$ denotes the Gibbs posterior with learning rate $\omega$, then we aim to choose $\omega$ such that the frequentist coverage probabilities of credible sets derived from $\Pi_n^\omega$ are approximately equal to the nominal level.  That is, for a desired significance level $\alpha \in (0,1)$, if $C_{\omega,\alpha}(X^n)$ denotes a $100(1-\alpha)$\% credible set from the Gibbs posterior $\Pi_n^\omega$, then the coverage probability is
\[ c_{\alpha}(\omega;P)=P\{C_{\omega,\alpha}(X^n)\ni \theta(P)\}, \]
i.e., the $P$-probability that $C_{\omega,\alpha}(X^n)$ contains $\theta(P)$.  Of course, if $P$ were known, then it would be possible to approximate the coverage probability using Monte Carlo and solve the equation, $c_{\alpha}(\omega;P)=1-\alpha$, using stochastic approximation \citep[e.g.,][]{robbinsmonro}.  Since $P$ is unknown in practice, \citet{syring2018calibrating} recommend a bootstrap version that replaces $P$ with the empirical distribution, $\PP_n$.  We use their {\em Gibbs posterior calibration} algorithm (see Algorithm~\ref{algo:gpc}) for choosing the learning rate, which performs well in our experiments below.

\begin{algorithm*}[t]
Fix a convergence tolerance $\epsilon >0$ and an initial value $\omega^{(0)}$ of the learning rate $\omega$. Take $B$ bootstrap samples $\tilde{X}_1^n,\ldots, \tilde{X}_B^n$ of size $n$. Set $t=0$ and do the following.
\begin{enumerate}
\item Construct $100(1-\alpha)\%$ credible set $C_{\omega^{(t)},\alpha}(\tilde{X}_b^n)$ for every $b=1,\dots,B$.
\item Evaluate the bootstrap estimate 
\[ \hat c_\alpha(\omega^{(t)}, \PP_n) = \frac{1}{B} \sum_{b=1}^B \one \bigl\{ C_{\omega^{(t)},\alpha}(\tilde X_b^n) \ni \hat\theta_n \bigr\} \]
of the empirical coverage probability $c_{\alpha}(\omega^{(t)},\PP_n)$.
\item If $\vert \hat{c}_{\alpha}(\omega^{(t)},\PP_n)-(1-\alpha) \vert < \epsilon$, then return $\omega^{(t)}$ as the output, else update $\omega ^{(t)}$ to $\omega^{(t+1)}$ as
\[ \omega^{(t+1)}=\omega^{(t)}+\kappa_t\{\hat{c}_{\alpha}(\omega^{(t)},\PP_n)-(1-\alpha)\}, \]
with $\kappa_t=(t+1)^{-0.51}$, set $t \leftarrow t+1$, and go back to Step 1.
\end{enumerate}
\caption{--- Gibbs Posterior Calibration \citep{syring2018calibrating}}
\label{algo:gpc}
\end{algorithm*}

For a quick visual illustration, consider a bivariate case, $d=2$.  Suppose we have $n=100$ samples from a bivariate normal distribution as in Example~1 in Section~\ref{SS:sim}.  Using a relatively flat $\nm_2(0, 10I_2)$ prior, and with the learning rate chosen according to Algorithm~\ref{algo:gpc}, samples from the corresponding Gibbs posterior distribution are shown in Figure~\ref{fig:illustration}(a).  The same is shown in Figure~\ref{fig:illustration}(b), except where the data are sampled from a bivariate Laplace distribution as in Example~2 of Section~\ref{SS:sim}.  In addition to the Gibbs posterior samples, we also display the 95\% credible region based on the normal approximation.  That is, we first compute the posterior mean $\bar\theta$ and covariance matrix $S$ based on the Monte Carlo samples, and then find the 95th percentile of the marginal posterior distribution for $\vartheta \mapsto (\vartheta - \bar\theta)^\top S^{-1} (\vartheta - \bar\theta)$, denoted by $r_{0.95}$.  Then the 95\% Gibbs posterior credible set is 
\[ \{\vartheta: (\vartheta - \bar\theta)^\top S^{-1} (\vartheta - \bar\theta) \leq r_{0.95}\}. \]
Similarly, we compute the 95\% confidence ellipse based on the asymptotic normality of the M-estimator/spatial median or, equivalently, the 95\% credible set from \eqref{eq:dumb.normal}, namely, 
\[ \{\vartheta: (\vartheta - \hat\theta_n)^\top \widehat\Gamma_n 
(\vartheta -\hat\theta_n) \leq \chi_{2;0.95}^2\}, \]
where $\chi^2_{2;.95}$ is the 95th percentile of the chi-square distribution with 2 degrees of freedom.  The boundaries of these two ellipses are overlaid on the plots of the Gibbs posterior samples.  Clearly, in both cases, the contours of the Gibbs posterior are not of the same shape as the M-estimator confidence ellipse, a consequence of the covariance mismatch.  However, by choosing the learning rate according to Algorithm~\ref{algo:gpc}, which is aiming to achieve the nominal 95\% frequentist coverage rate, the Gibbs posterior credible ellipse is stretched to roughly match the confidence ellipse in the direction in which it is widest.  And since the confidence ellipse achieves the nominal frequentist coverage probability, at least asymptotically, the Gibbs posterior credible ellipse will too.  Of course, there is some loss of efficiency due to the covariance mismatch---which is the price one pays for a model-free posterior distribution---but, as the simulation results in Section~\ref{SS:sim} show, this loss of efficiency is not severe.  In fact, in some cases, the Gibbs posterior credible regions are more efficient than those of other Bayesian methods.  


\begin{figure}[t]
\centering     
\subfigure[Bivariate normal]{\scalebox{0.75}{\includegraphics{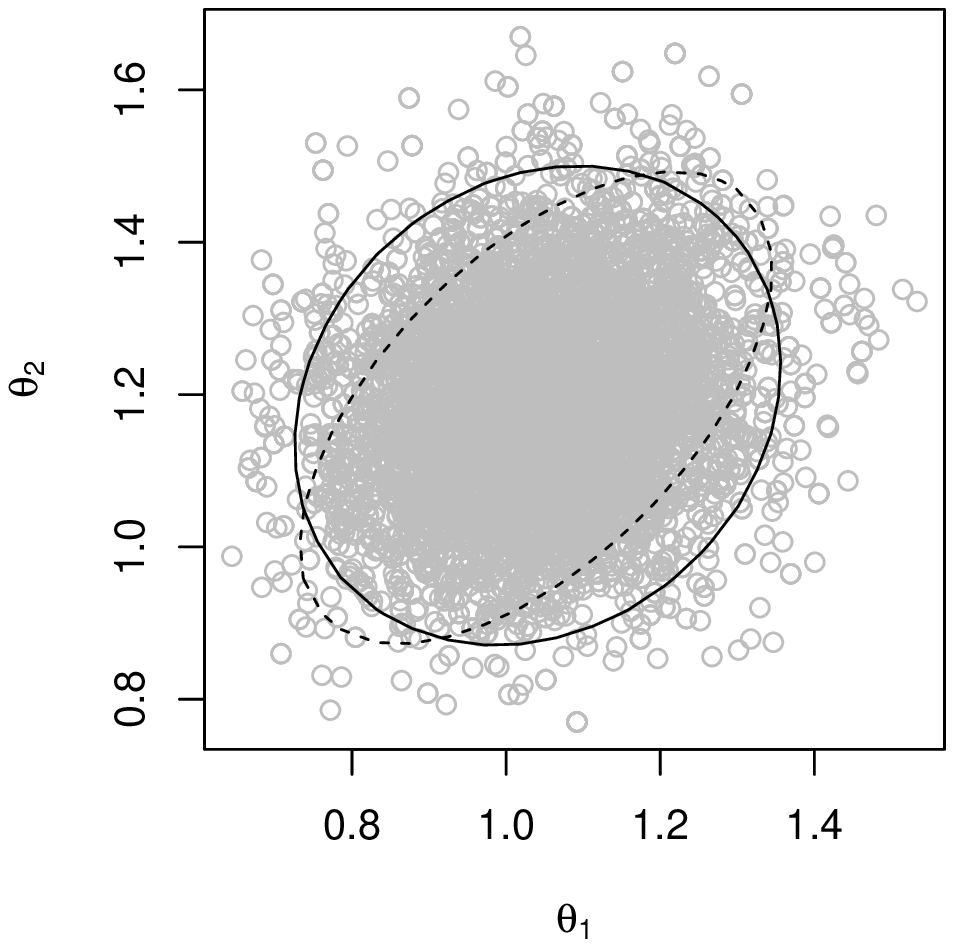}}}
\subfigure[Bivariate Laplace]{\scalebox{0.75}{\includegraphics{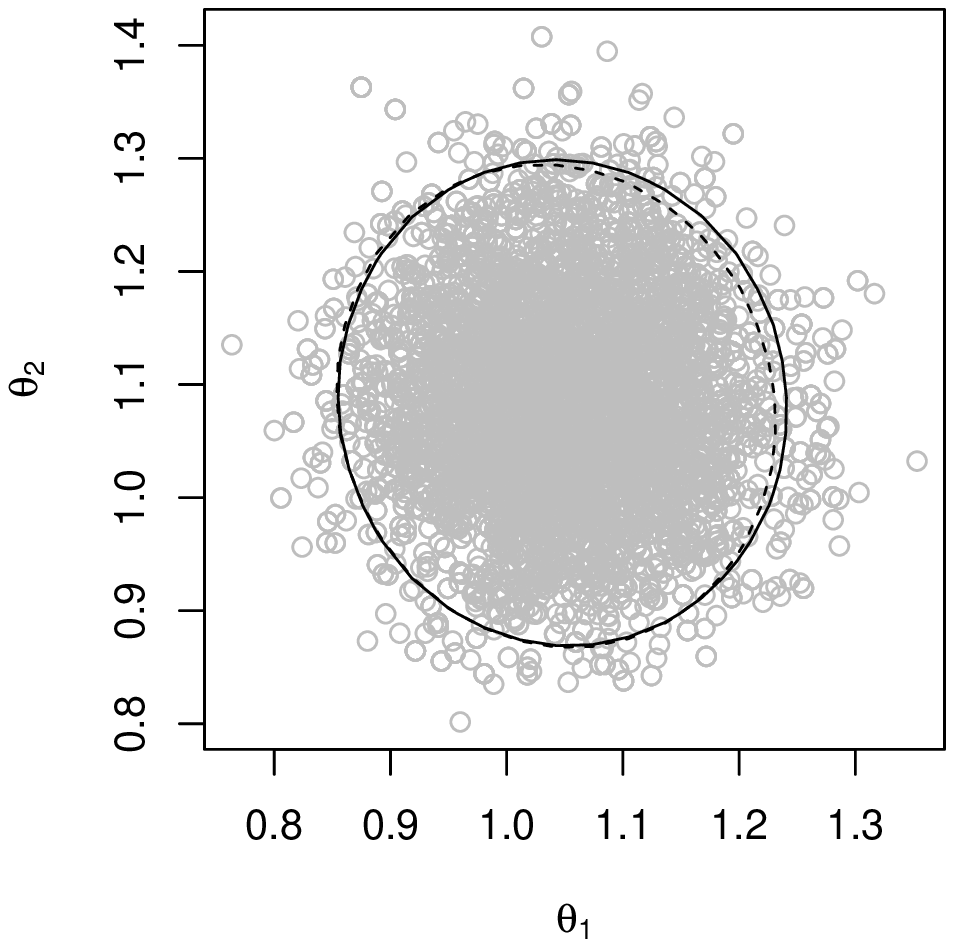}}}
\caption{Gibbs posterior samples (gray), with learning rate chosen according to Algorithm~\ref{algo:gpc}, along with the 95\% posterior credible region (solid) and the corresponding M-estimator confidence region (dashed).}
\label{fig:illustration}
\end{figure}

\section{Numerical results}

\subsection{Simulation study}
\label{SS:sim}

In this section, we illustrate the finite sample performance of the Gibbs posterior of bivariate $\ell_1$-medians and quantiles, i.e., for $d=2$. For the median, we would like to compare the Gibbs posterior's performance to that of both parametric and non-parametric Bayesian methods in situations when the components of the vector are correlated, and when the data has outliers. Aside from the $d$-variate normal distribution $\nm_d(\mu, \Sigma)$ with mean vector $\mu$ and covariance matrix $\Sigma$, we also consider a $d$-variate Laplace distribution, denoted by $\lap_d(\mu,\Sigma)$, with location vector $\mu$ and dispersion matrix $\Sigma$, with density for the standardized version, with $\mu=0$ and $\Sigma=I_d$, 
\[ f(x) \propto \|x\|_2^{-(d-1)/2} e^{-2^{3/2} \|x\|_2}, \quad x \in \RR^d. \]
Also, $\gam_d(s,r,V)$ denotes a $d$-variate gamma distribution with shape $s$, rate $r$, and correlation matrix $V$, constructed using a Gaussian copula \citep{xue2000multivariate}.  The three specific examples we consider are as follows, each with sample size $n=100$.  

\begin{example}
$P = \nm_2(\mu,\Sigma)$, where $\mu=(1,1)^\top$, $\Sigma_{11}=\Sigma_{22}=1$ and $\Sigma_{12}=0.7$.  
\end{example}

\begin{example}
$P = \lap_2(\mu, \Sigma)$, where $\mu=(1,1)^\top$ and $\Sigma=I_2$. 
\end{example}

\begin{example}
$P = \gam_2(1, 1, V)$, where $V_{11}=V_{22}=1$ and $V_{12}=0.5$.
\end{example}

In each case, for the Gibbs posterior, we consider a bivariate normal prior for $\theta$, namely, $\nm_2((0,0)^\top,10 I_2)$.  We opt for a relatively non-informative prior, sacrificing one of the key benefits of the Gibbs posterior (Section~\ref{SS:def}) so that the performance of our Gibbs posterior is not confounded with the use of an informative prior. We also use a Metropolis--Hastings algorithm with transition kernel $Q(y \mid x)= \nm_2(y \mid x, 0.01 I_2)$ for drawing samples from the posterior distribution. The relevant summaries would be size and frequentist coverage of the 95\% credible ellipses. 

We compare the Gibbs posterior performance to that of parametric and non-parametric Bayesian methods. For the parametric Bayes model, we consider 
\[ (X_1,\ldots,X_n) \mid \theta \iid \nm_2(\theta, \sigma^2 I_2) \quad \text{and} \quad \theta \sim \nm_2((0,0)^\top, 10 I_2) \quad \sigma^{-2} \sim \gam(1,1). \]
This is quite simple, and we can use a Gibbs sampler for posterior inference.  While the model and corresponding analysis is simple, the concern is potential model misspecification bias.  
As a more robust alternative, one can consider a nonparametric Bayesian formulation.  As suggested by \citet{bhattacharya2019bayesian}, assume 
\[ (X_1,\ldots,X_n) \mid P \iid P \quad \text{and} \quad P \sim \dpp(\alpha), \]
where $\dpp(\alpha)$ denotes a Dirichlet process distribution with base or centering measure $\alpha$ \citep[e.g.,][Ch.~4]{ghosal2017fundamentals}.  Here we choose $\alpha = 2 \times \nm_2((0,0)^\top, I_2)$.  It is well known that the Dirichlet process prior is conjugate, so the posterior distribution for $P$, given $X^n$, is also a Dirichlet process, which is relatively simple to work with.  However, the quantity of interest is $\theta=\theta(P)$, a functional of $P$, so some non-trivial marginalization is required.  Specifically, we sample $P$ from the Dirichlet process posterior distribution, and then evaluate $\theta$ as the minimizer of $\xi \mapsto P\|X-\xi\|_r$. \citet{bhattacharya2019bayesian} establish that this nonparametric Bayes marginal posterior for $\theta$ is asymptotically equivalent to the Gaussian distribution in \eqref{eq:dumb.normal}, i.e., no covariance mismatch, so must be relatively close to the best possible frequentist solution.   

%

To compare the four methods described above, we consider a measure of size and also the frequentist coverage probability of the 95\% posterior credible sets.  Suppose that we have samples $\theta_1,\ldots,\theta_M$ from any one of the three posterior distributions, where $M=5000$ is the Monte Carlo sample size.  The 95\% credible set is given by 
\[ \{\vartheta: (\vartheta - \bar\theta)^\top S^{-1} (\vartheta - \bar\theta) \leq r_{0.95}\}, \]
where $\bar{\theta}$ and $S$ are the Monte Carlo sample mean and covariance matrix, respectively.  The coverage probability is defined as usual and, as a measure of the credible set's size, we use $\vert S \vert r_{0.95}^d$, with $\vert S \vert$ denoting the determinant of the matrix $S$.  Table~\ref{tab:sim} summarizes the size and coverage probability over 2000 replications for each of the three examples.  It can be seen that the Gibbs posterior performs well compared to the parametric and nonparametric Bayesian approaches in terms of both size and coverage.  Indeed, the Gibbs posterior has at least the nominal 95\% coverage in every scenario, while the parametric and nonparametric Bayesian credible sets occasionally miss the target coverage, especially the parametric solution. 
The coverage performance of the Gibbs credible regions may not be surprising, given that we tuned the learning rate to achieve the nominal coverage.  However, it is interesting to see that, despite the covariance mismatch, the coverage guarantees do not come with any perceptible loss of efficiency.  In fact, in some cases, especially in Example~3, the Gibbs posterior credible sets might be more efficient.

\begin{table}[t]


\begin{center}
 \begin{tabular}{ccccc} 
 \hline
Example  & $r$ & Gibbs & PBayes & NPBayes\\ \hline 
1 & $2$ & 0.950 (0.42) & 0.900 (0.33) & 0.975 (0.37) \\ 
 & $3$ & 0.955 (0.40) & 0.900 (0.33) & 0.975 (0.38) \\ 
2 & $2$ & 0.950 (0.21) & 0.925 (0.34) &0.920 (0.19) \\ 
 & $3$ & 0.960 (0.20) & 0.925 (0.34) & 0.970 (0.21) \\ 
3 & $2$ & 0.950 (0.31) & 0.926 (0.35) & 0.950 (0.40) \\ 
 & $3$ & 0.965 (0.21) & 0.910 (0.35) & 0.949 (0.27) \\ 
\hline 
\end{tabular}
\caption {Estimated coverage probability and mean size (in parentheses) of $95\%$ credible ellipses of the $\ell_1$-median (with $\ell_2$ and $\ell_3$ norms) for the Gibbs, parametric Bayes, and nonparametric Bayes posterior distributions.}
\label{tab:sim} 
\end{center}
\end{table}

We also consider inference on a general $u^\text{th}$ geometric quantile other than the $\ell_1$-median, for $u=(0.2,0.3)$.  While one can imagine the $\ell_1$-median as a type of location parameter, the general geometric quantile has no such interpretation.  This implies that there is no genuine likelihood function for this unknown, which is precisely the case in which the Gibbs framework is most valuable.  Here we compare the performance of the Gibbs posterior with the nonparametric Bayesian approach described above, again in terms of coverage probability and credible region size.  The results are presented in Table~\ref{tab:sim2}.  Note here that the Gibbs posterior achieves or exceeds the target 0.95 coverage probability while being more (or at least no less) efficient than the nonparametric Bayes solution.  It is worth noting here that we have used a relatively flat prior for our illustration.  But when relevant prior information about the quantile is available, this can readily be incorporated into the Gibbs posterior because it deals directly with the quantile, which would lead to additional gains in efficiency.  The nonparametric Bayes formulation, on the other hand, focuses on the distribution $P$ and, hence, works only indirectly with $\theta=\theta(P)$, so it is not at all clear how to incorporate prior information about $\theta$ to achieve those same efficiency gains.  

\begin{table}[t]
\begin{center}
 \begin{tabular}{ccc} 
 \hline
Example & Gibbs & NPBayes\\ \hline 
1 & 0.990 (0.56) & 0.955 (0.71) \\ 
2 & 0.950 (0.43) & 0.925 (0.44) \\ 
3 & 0.950 (0.31) & 0.945 (0.41) \\
\hline 
\end{tabular}
\caption {Estimated coverage probability and mean size (in parentheses) of $95\%$ credible ellipses of $u=(0.2,0.3)$th geometric quantile (with $\ell_2$ norm) for the Gibbs and nonparametric Bayes posterior distributions.}
\label{tab:sim2} 
\end{center}
\end{table}

\subsection{Real data analysis}
\label{SS:real}

The Egyptian skulls dataset \citep{hand1994} consists of $d=4$ measurements---namely, maximal breadth, basibregmatic height, basialveolar length, and nasal height---taken on $n=150$ ancient Egyptian skulls from five time epochs; these data are available in the {\tt HSAUR} package in R. The mean effect of time was removed by fitting a linear model and extracting the residuals, and we take these 4-dimensional residual vectors as our data $X^n$.  A first thought would be to assume multivariate normality and carry out a standard analysis.  However, formal tests of normality are conflicting: marginal tests of normality reject while tests of multivariate normality do not reject \citep[e.g.,][]{dpmgof}.  So, we proceed with the construction of a Gibbs posterior that does not require us to decide about normality in this difficult case.  

For simplicity, here we will focus on the median of the 4-dimensional distribution, although other quantiles could be handled similarly.  Since these data already have time trends removed, we expect that the distribution's center should be roughly near the origin, so we take a normal prior with zero mean, but with covariance matrix $10 I_4$.  Since we are not at the $n \to \infty$ limit, we expect some non-elliptical shape in the Gibbs posterior, and the goal of this analysis is to investigate that shape.  Figure~\ref{fig:skull} summarizes the marginal and pairwise Gibbs posterior distributions after scaling the learning rate according to Algorithm~\ref{algo:gpc}.  As expected, even on these limited low-dimensional summaries, the Gibbs posterior does not appear to be exactly normal, but the results are quite reasonable.  For comparison, we also show the marginal plug-in densities and 95\% pairwise confidence ellipses based on the asymptotically normal sampling distribution of the M-estimator, the sample spatial median.  Clearly, these margins of the Gibbs posterior are centered in roughly the correct place but, most importantly, and thanks to the calibration framework in Algorithm~\ref{algo:gpc}, the Gibbs posterior spread tends to be wider in some directions than that of the M-estimator sampling distribution.  While some might view this wider spread as an indication of some inefficiency, we believe the wider spread is necessary for valid uncertainty quantification in finite samples, not just in the idealistic $n \to \infty$ case.  

\begin{figure}[t]
\begin{center}
\scalebox{0.75}{\includegraphics{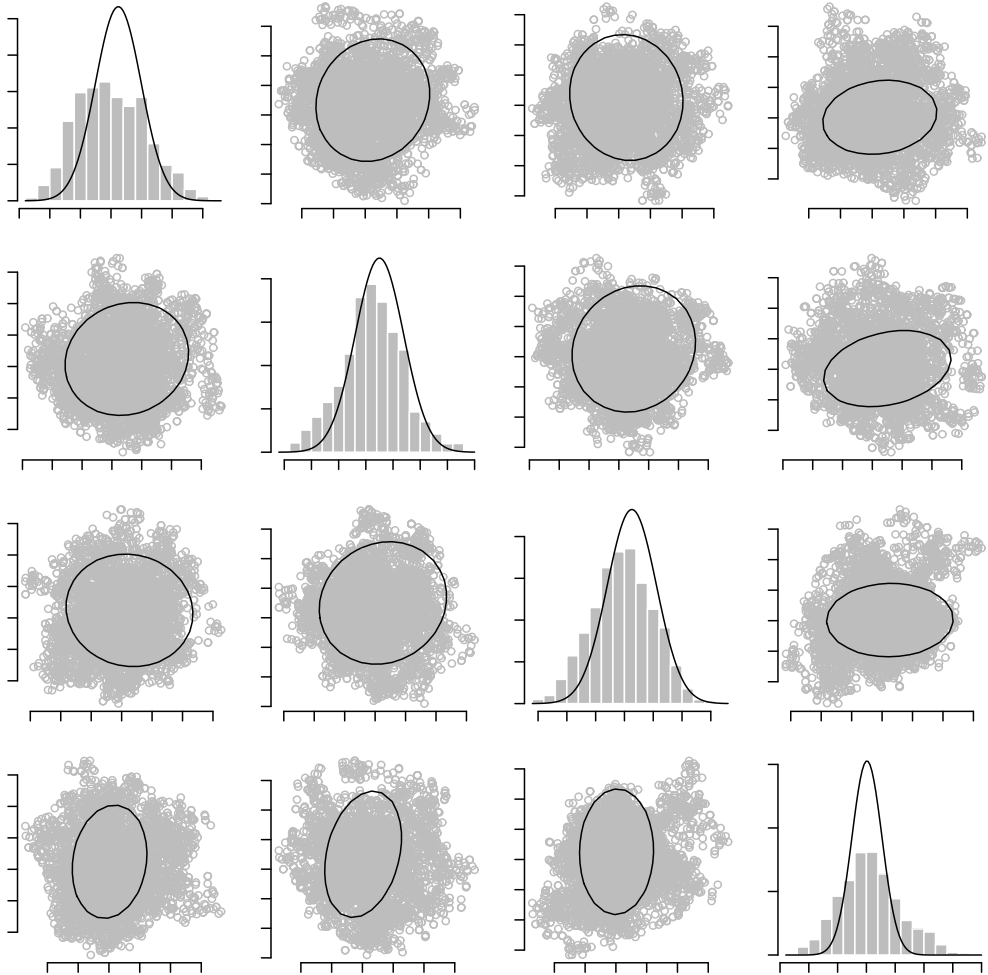}}
\end{center}
\caption{Marginal and pairwise Gibbs posterior distributions (gray) for the 4-dimensional spatial median of the Egyptian skull data described in Section~\ref{SS:real}.  Overlaid (black) are approximate marginal sampling distributions and 95\% pairwise confidence ellipses based on asymptotic normality of the M-estimator.}
\label{fig:skull}
\end{figure}

To compare the performance of our Gibbs posterior with a Bayes solution that assumes normality, we look at the posterior of the empirical risk evaluated on a held-out testing set.  That is, we split the data into a training and testing set---the first 100 samples are training and the last 50 are testing---construct a posterior distribution for the median $\theta$ using the training data, then evaluate the corresponding posterior distribution for $R_{\text{test}}(\theta)$, where $R_{\text{test}}$ is the empirical risk function in \eqref{eq:erisk} based on the testing data set only.  Figure~\ref{fig:risktest} plots (kernel density estimates of) the posterior distribution of log empirical risk difference, 
\begin{equation}
\label{eq:log.test.risk}
\log\{ R_{\text{test}}(\theta) - \textstyle \min_\vartheta R_{\text{test}}(\vartheta)\}, 
\end{equation}
based on the testing data.  We follow the same Gibbs formulation as above; for the Bayes solution, we assume $P = \nm_4(\theta, \Sigma)$ and use the conjugate normal--inverse Wishart prior for $(\theta,\Sigma)$.  The figure shows that the Gibbs posterior distribution of the log empirical risk difference is centered to left of that for Bayes, which is an indication that the former is more concentrated around $\theta$ values that make the out-of-sample risk small than the latter.  Since the training and testing data are not fundamentally different, this suggests that there is some bias created by the assumption of multivariate normality made by the parametric Bayesian solution.  The Gibbs posterior, however, is apparently not susceptible to this model misspecification bias.  

\begin{figure}[t]
\begin{center}
\scalebox{0.7}{\includegraphics{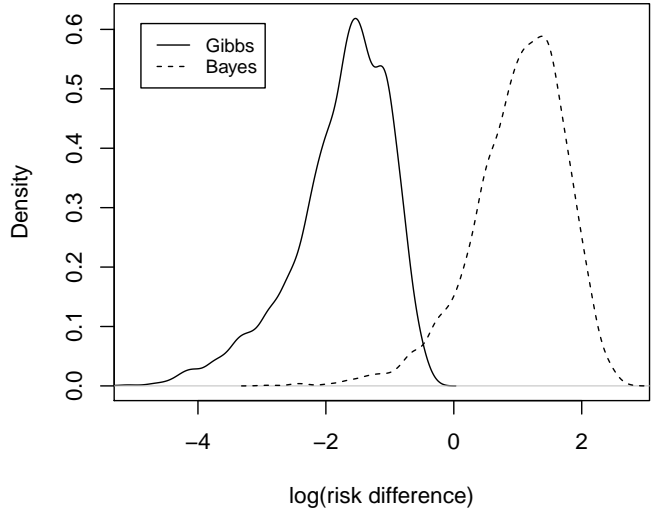}}
\end{center}
\caption{Plots of (kernel density estimates of) the Gibbs and Bayes posterior distribution for the log risk difference \eqref{eq:log.test.risk} based on the testing data.}
\label{fig:risktest}
\end{figure}

\section{Concluding remarks}
\label{S:discuss}

In this paper, we have studied multivariate medians and quantiles in a Gibbs posterior framework. Our approach does not need a model and is free from the potential issues that may arise in a model-based parametric Bayesian approach, in particular, model misspecification bias. The Gibbs posterior is simple to use and is also theoretically justified in the sense that the posterior concentrates around the true multivariate quantile at the optimal $n^{-1/2}$ rate, and has a Bernstein--von Mises property, i.e., it can be approximated by a suitable Gaussian distribution centered at the sample spatial median.  We also pointed out that this Gaussian approximation holds in other problems, not just multivariate quantiles, provided that the empirical risk satisfies a version of the local asymptotic normality property.  

A unique feature of the Gibbs posterior is its dependence on the choice of learning rate.  On one hand, this dependence might seem like a disadvantage, since the learning rate is not determined by the context of the problem, and there is no universally accepted choice.  On the other hand, as we argued here, being able to choose the learning rate is an advantage in the sense that it provides the flexibility necessary to at least partially correct for the covariance mismatch in the Gaussian approximation discussed in Section~\ref{SS:theory}.  Here we recommend the data-driven learning rate selection procedure of \citet{syring2018calibrating}, summarized in Algorithm~\ref{algo:gpc}, as it aims to set $\omega$ so that the Gibbs posterior credible region achieves the nominal frequentist coverage probability.  

While the learning rate selection procedure described in Algorithm~\ref{algo:gpc} works well empirically, there are still some opportunities for improvement and some unanswered questions.  The main disadvantage of this strategy is having to do multiple Monte Carlo runs on each bootstrap sample; this could be improved by carrying out some of the steps in parallel.  In terms of open questions, so far there is no theory to support the claim that choosing the learning rate according to Algorithm~\ref{algo:gpc} will, as advertised, produce credible sets that achieve the nominal frequentist coverage.  \citet{syring2018calibrating} argue that, under regularity conditions like those here that yield root-$n$ convergence rates, there exists a value of $\omega$ such that the Gibbs credible regions achieve the nominal coverage.  Then Algorithm~\ref{algo:gpc} uses standard simulation-based techniques---bootstrap, Monte Carlo, and stochastic approximation---to find this solution.  There is no reason to doubt that it would work, and the empirical results confirm this.  Unfortunately, these three standard simulation-based techniques working in tandem make the algorithm quite difficult to analyze theoretically.  That makes it an interesting open problem.


There are a couple of possible extensions of the work presented herein:
\begin{itemize}
\item It would be interesting to explore cases where the dimension $d$ exceeds the sample size, i.e., a so-called ``high-dimensional setting,'' with $d \gg n$.  For such cases, we would need to assume some low-dimensional structure in the high-dimensional $\theta^\star$, and then specify a prior distribution that would encourage this structure.  Sparsity-inducing priors \citep[e.g.,][]{pas.szabo.vaart.uq, castillo.vaart.2012, ebcvg} have been popular in recent years, and one of the examples in \citet{gibbs.general} shows that this kind of sparsity can be readily handled within the Gibbs framework, but the details for a sparse, high-dimensional multivariate quantile have yet to be worked out.
\vspace{-2mm}
\item The Gibbs posterior approach can also be used in multivariate quantile regression. Consider a linear regression set-up with a $d$-variate response vector $y$ and a $q$-dimensional regressor $x$ satisfying the linear model $y=\beta^\top x +e$, with $\beta$ a $q \times d$ matrix of regression coefficients.  For $u \in B_2^{(d)}$, and a sample $(x_i,y_i)$ of response and regressor pairs, the $u^\text{th}$ sample geometric quantile of $y$ given $x$ is obtained as
\[ Q_{y \mid x}(u)=\arg \min_{\beta \in \RR^d} \frac{1}{n}\sum_{i=1}^n\{\|y_i-\beta^\top x_i \|_r + \langle u, y_i-\beta^\top x_i \rangle\}. \]
In a typical Bayesian approach, we would have to choose a model for the errors such that its quantile agrees with the target quantile; of course, there are many such models, so having to make such a choice puts the data analyst at risk of model misspecification bias.  On the other hand, it is easy to formulate a Gibbs posterior framework, which is free of such modeling and the associated risks.
We expect that the theoretical results for the Gibbs posterior presented herein would carry over to this more general setting, but we have yet to verify this conjecture.  

\end{itemize}

\section*{Acknowledgments}

Thanks go to two anonymous reviewers for their helpful comments that lead to a number of improvements to the manuscript, both in presentation and strength of results.  This work is partially supported by the U.S.~National Science Foundation, DMS--1811802.

\appendix

\section{Proofs}

\subsection{Preliminary results}

Recall that $\ell_{\theta}(x) = \|x -\theta \|_r +\langle u, x-\theta \rangle$, and $R(\theta)=P\ell_{\theta}$, and the second derivative matrix is given by $V_{\theta^{\star}}$.  First, we want to bound
\begin{equation}
    \inf_{\|\theta-\theta^{\star}\|_2 >\delta} R(\theta)-R(\theta^{\star}).
    \label{eq11}
\end{equation}

\begin{lemma}
\label{lem:separation}
Under Assumptions~\ref{asmp:density1}--\ref{asmp:density2}, there exists a constant $C>0$ such that \eqref{eq11} is lower-bounded by $C\delta^2$ for all sufficiently small $\delta>0$.
\end{lemma}

\begin{proof}
The function $R$ is twice-differentiable at $\theta^{\star}$, and the second derivative matrix is given by $V_{\theta^{\star}}=\ddot R(\theta^\star)=\int v_{\theta^{\star}}(x) \, P(dx)$, where
\[ v_{\theta^{\star}}(x)=\frac{r-1}{\Vert x-\theta^{\star}\Vert_r}\Big[\diag \left(\frac{\vert x_1-\theta^{\star}_1\vert^{r-2}}{\Vert x-\theta^{\star}\Vert_r^{r-2}},\dots,\frac{\vert x_d-\theta^{\star}_d\vert^{r-2}}{\Vert x-\theta^{\star}\Vert_r^{r-2}}\right) - \frac{yy^\top}{\Vert x-\theta^{\star}\Vert_r^{2(r-1)}}\Big], \]
with $y=y(x,\theta^\star)$ being equal to
\[ y = \bigl( \vert x_1-\theta^{\star}_1 \vert^{r-1}\sign(x_1-\theta^{\star}_1),\dots,\vert x_d-\theta_d^{\star} \vert^{r-1}\sign(x_d-\theta_d^{\star}) \bigr)^\top. \]
The existence of $V_{\theta^{\star}}$ can be verified using Assumption~\ref{asmp:density2}, i.e., for a fixed $\theta \in \RR^d$, if $P$ has a density $p$ that is bounded on compact subsets of $\RR^d$, then the expectation of $\Vert X-\theta \Vert_r^{-1}$ is finite. Since the first derivative vanishes at $\theta^{\star}$, the Taylor expansion takes the form
\begin{equation*}
R(\theta)-R(\theta^{\star})=\tfrac{1}{2}(\theta-\theta^{\star})^\top V_{\theta^{\star}}(\theta-\theta^{\star})+o(\Vert \theta -\theta^{\star}\Vert_2^2).
\end{equation*}
Since $V_{\theta^{\star}}$ is positive definite, the proof follows by taking $C=\frac12 \lambda_\text{min}(V_{\theta^\star})$, where $\lambda_{\text{min}}(M)$ returns the smallest eigenvalue of the matrix $M$.
\end{proof}

Now, rewrite the Gibbs posterior distribution as
\begin{equation}
\label{eq:post.app}
\Pi_n(A) = \frac{N_n(A)}{D_n} = \frac{\int_A e^{-\omega n\{R_n(\theta)-R_n(\theta^{\star})\}} \, \Pi(d\theta)} {\int e^{-\omega n\{R_n(\theta)-R_n(\theta^{\star})\}} \, \Pi(d\theta)}, \quad A \subseteq \RR^d. 
\end{equation}
Below we will investigate the limiting behavior of $\Pi_n(A)$ for two kinds of sets $A$: the first is for a fixed $A=\XX^c$, where $\XX$ is defined in Assumption~\ref{asmp:density1}, and the second is for a sequence of suitably shrinking sets $A_n$ to be defined below.  

\begin{lemma}
\label{lem:consistency}
Under Assumptions~\ref{asmp:density1}--\ref{asmp:density2}, for any $\omega > 0$, the Gibbs posterior $\Pi_n$ satisfies $P^n \Pi_n(\XX^c) = o(1)$ as $n \to \infty$.
\end{lemma}

\begin{proof}
Without loss of generality, assume $\XX = \{x: \|x-\theta^\star\|_2 \leq K\}$ for $K > 0$. We start with the numerator $N_n(\XX^c)$. By the law of large numbers, for any fixed unit vector $u$, 
\[ R_n(\theta^\star + Ku) - R_n(\theta^\star) \to R(\theta^\star + Ku) - R(\theta^\star), \quad \text{in $P^n$-probability}. \]
Moreover, is easy to check that $\theta \mapsto R_n(\theta)$ is almost surely convex, so it follows from Lemma~1 in \citet{hjort.pollard.1993} that
\[ \inf_u \{R_n(\theta^\star + Ku) - R_n(\theta^\star)\} \to \Delta \quad \text{in $P^n$-probability}, \]
where the infimum is over all unit vectors $u$, and $\Delta = \inf_u \{R(\theta^\star + Ku) - R(\theta^\star)\}$ is strictly positive.  Then the event $\{N_n(\XX^c) \leq e^{-\omega n (\Delta/2)}\}$ is implied by 
\[ \inf_u \{R_n(\theta^\star + Ku) - R_n(\theta^\star)\} > \Delta/2, \]
which has $P^n$-probability converging to 1.  

For the denominator $D_n$, Lemma~\ref{lem:den} below establishes that $D_n > \frac12 \eps_n^d e^{-\omega(1+c_n)}$, with $P^n$-probability converging to 1, where $c_n$ is such that $c_n \to \infty$ arbitrarily slowly.  Then, with $P^n$-probability converging to 1,
\[ \Pi_n(\XX^c) = \frac{N_n(\XX^c)}{D_n} \leq 2e^{-n\omega \Delta / 2 + \omega(1+c_n) + d \log \eps_n}. \]
We are free to choose $c_n = o(n)$, so the upper bound is vanishing.  Therefore, we can conclude that $\Pi_n(\XX^c) \to 0$ in probability.  Since $\Pi_n(\XX^c)$ bounded and converges to 0 in probability, the claim follows from the dominated convergence theorem.
\end{proof}

\begin{lemma}
\label{lem:donsker}
Define the empirical process $\GG_n f = n^{1/2}(\PP_n f - P f)$ and, for the compact $\XX$ in Lemma~\ref{lem:consistency}, define the event 
\[ \mathscr{B}_n = \bigl\{X^n: |\GG_n(\ell_\theta - \ell_{\theta^\star})| > b_n \|\theta-\theta^\star\|_2 \;\; \text{for some $\theta \in \XX$} \bigr\}, \]
where $b_n > 0$ is any divergent sequence $b_n \to \infty$.  Then $P^n(\mathscr{B}_n) = o(1)$ as $n \to \infty$. 
\end{lemma}

\begin{proof}
Since the loss function is Lipschitz and the domain of $\theta \mapsto \ell_\theta - \ell_{\theta^\star}$ is restricted to the compact $\XX$, it follows from Example 19.7 in \citet{van2000asymptotic} that $\{\ell_\theta - \ell_{\theta^\star}: \theta \in \XX\}$ is a Donsker class.  This implies weak convergence and, in particular, stochastic equicontinuity of the empirical process $\theta \mapsto \GG_n(\ell_\theta - \ell_{\theta^\star})$ on $\XX$.  Therefore, the uniform-on-$\XX$ behavior of the empirical process is controlled by its behavior at finitely many points in $\XX$.  At any fixed $\theta \in \XX$, Bernstein's inequality gives 
\[ P^n\bigl\{ |\GG_n(\ell_\theta - \ell_{\theta^\star})| > t \bigr\} \leq 2 \exp\Bigl\{- \frac{\frac12 t^2}{\|\theta-\theta^\star\|_2^2 + t \|\theta-\theta^\star\|_2} \Bigr\}. \]
If we take $t=\{2b_n \|\theta-\theta^\star\|_2^2\}^{1/2}$, then the right-hand side of the above display is 
\[ \exp\Bigl\{ -\frac{b_n \|\theta-\theta^\star\|_2^2}{\|\theta-\theta^\star\|_2^2 + (2b_n)^{1/2} \|\theta-\theta^\star\|_2^2} \Bigr\} \leq \exp\{-\text{constant} \times b_n^{1/2}\}. \]
Since $b_n \to \infty$, the upper bound is vanishing, so we get 
\[ P^n\bigl\{|\GG_n(\ell_\theta - \ell_{\theta^\star})| > b_n \|\theta-\theta^\star\|_2 \bigr\} = o(1), \quad \text{for each fixed $\theta \in \XX$}. \]
Then this, stochastic equicontinuity, and the union bound gives the desired result.  
\end{proof}

Recall that the events of interest in Theorem~\ref{thm:rate} are given by $A_n = \{\theta: \|\theta-\theta^\star\|_2 > a_n \eps_n\}$, where $\eps_n = n^{-1/2}$ and $a_n \to \infty$ is arbitrary.  To prove that theorem, we need to show that $\Pi_n(A_n) \to 0$ in expectation.  Our strategy is to find a lower bound on $D_n$ and an upper bound on $N_n(A_n)$, both defined in \eqref{eq:post.app}, such that the ratio of these two bounds is vanishing.  The next two lemmas accomplish each these two goals in turn.  

\begin{lemma}
\label{lem:den}
Let $c_n > 0$ be any divergent sequence, and define the event 
\[ \mathscr{C}_n = \{X^n: D_n \leq \tfrac12 \eps_n^d e^{-\omega(1+c_n)}\}. \]
Then under the conditions of Theorem~\ref{thm:rate}, $P^n(\mathscr{C}_n) = o(1)$ as $n \to \infty$. 
\end{lemma}

\begin{proof}
Define the set $K_n = \{\theta: m(\theta) \vee v(\theta) \leq \eps_n^2\}$, where $a \vee b = \max(a,b)$, and 
\[ m(\theta) = R(\theta) - R(\theta^\star) \quad \text{and} \quad v(\theta) = P\{(\ell_\theta-\ell_{\theta^\star}) - m(\theta)\}^2, \]
are the mean and variance of the loss difference, respectively.  From the Lipschitz property of the loss and the Taylor approximation in the proof of Lemma~\ref{lem:separation}, it follows that 
\[ \|\theta-\theta^\star\|_2 \lesssim \eps_n \implies m(\theta) \vee v(\theta) \leq \eps_n^2. \]
Therefore, by Assumption~\ref{asmp:prior}, 
\[ \Pi(K_n) \geq \Pi(\{\theta: \|\theta-\theta^\star\|_2 \lesssim \eps_n\}) \gtrsim \eps_n^d. \]
Next, define a standardized version of the empirical risk difference, i.e., 
\[ Z_n(\theta) = \frac{n\{R_n(\theta) - R_n(\theta^\star)\} -  nm(\theta)}{\{ nv(\theta)\}^{1/2}}. \]
This is a function of both $\theta$ and the data $X^n$, so define the upper level sets 
\[ \Z_n = \{(\theta, X^n): |Z_n(\theta)| \geq c_n \}, \]
where $c_n$ is the sequence in the lemma statement. Also define the cross-sections 
\[ \Z_n(\theta) = \{X^n: (\theta, X^n) \in \Z_n\} \quad \text{and} \quad \Z_n(X^n) = \{\theta: (\theta, X^n) \in \Z_n\}. \]
Since we have 
\[ n\{R_n(\theta) - R_n(\theta^\star)\} = n m(\theta) + \{n v(\theta)\}^{1/2} Z_n(\theta), \]
and $m$, $v$, and $Z_n$ are suitably bounded on $K_n \cap \Z_n(X^n)^c$, we immediately get 
\[ D_n \geq \int_{K_n \cap \Z_n(X^n)^c} e^{-\omega n m(\theta) - \omega \{ nv(\theta)\}^{1/2} Z_n(\theta)} \, \Pi(d\theta) \geq e^{-\omega(1+c_n)} \Pi\{K_n \cap \Z_n(X^n)^c\}, \]
using the fact that $n\eps_n^2=1$.  From this lower bound, we get
\begin{align*}
P^n\{D_n \leq \tfrac12 \Pi(K_n) e^{-\omega(1+c_n)}\} & \leq P^n\bigl[ e^{-\omega (1+c_n)} \Pi\{K_n \cap \Z_n(X^n)^c\} \leq \tfrac12 \Pi(K_n) e^{-\omega(1+c_n)} \bigr] \\
& = P^n\bigl[ \Pi\{K_n \cap \Z_n(X^n)\} \geq \tfrac12 \Pi(K_n) \bigr] \\
& \leq \frac{2 P^n \Pi\{K_n \cap \Z_n(X^n)\}}{\Pi(K_n)},
\end{align*}
where the last line is by Markov's inequality.  Now use Fubini's theorem: 
\begin{align*}
P^n \Pi\{K_n \cap \Z_n(x^n)\} & = \int \int 1\{\theta \in K_n \cap \Z_n(x^n)\} \, \Pi(d\theta) \, P^n(dx^n) \\
& = \int \int 1\{\theta \in K_n\} \, 1\{\theta \in \Z_n(x^n)\} \, P^n(dx^n) \, \Pi(d\theta) \\
& = \int_{K_n} P^n\{\Z_n(\theta)\} \, \Pi(d\theta).
\end{align*}
By Chebyshev's inequality, $P^n\{\Z_n(\theta)\} \leq c_n^{-2}$, and hence  
\[ P^n\{D_n \leq \tfrac12 \Pi(K_n) e^{-\omega(1+c_n)}\} \leq 2 c_n^{-2}. \]
Putting everything together, since $\Pi(K_n) \gtrsim \eps_n^d$, we have that 
\[ P^n(\mathscr{C}_n) = P^n(D_n \leq \tfrac12 \eps_n^d e^{-\omega(1+c_n)}) \leq 2 c_n^{-2} = o(1), \quad n \to \infty. \qedhere \]
\end{proof}


\begin{lemma}
\label{lem:num}
Let $\XX$ and $\mathscr{B}_n$ be as in Lemma~\ref{lem:donsker}, with the sequence $b_n$ in $\mathscr{B}_n$ such that $b_n = o(a_n)$.  Then there exists a constant $k > 0$ such that
\[ P^n \{N_n(A_n \cap \XX) \, 1(\mathscr{B}_n^c)\} \lesssim \eps_n^d a_n^d e^{-\omega k a_n^2}, \quad \text{for all large $n$}. \]
\end{lemma}

\begin{proof}
For the empirical process $\GG_n$ defined above, write 
\[ R_n(\theta) - R_n(\theta^\star) = R(\theta) - R(\theta^\star) + n^{-1/2} \GG_n(\ell_\theta - \ell_{\theta^\star}). \]
Then the Gibbs posterior numerator at $A_n \cap \XX$ can be decomposed as a sum of integrals over ``shells'' as follows:
\begin{align*}
N_n(A_n \cap \XX) & = \sum_{t=1}^{T_n} \int_{ta_n \eps_n < \|\theta-\theta^\star\|_2 < (t+1)a_n \eps_n} e^{-\omega n\{R_n(\theta) - R_n(\theta^\star)\}} \, \Pi(d\theta) \\
& \leq \sum_{t=1}^{T_n} e^{-C\omega t^2 a_n^2} \int_{\|\theta-\theta^\star\|_2 < (t+1)a_n \eps_n} e^{-\omega n^{1/2} \GG_n(\ell_\theta - \ell_{\theta^\star})} \, \Pi(d\theta),
\end{align*}
where $C > 0$ is as in Lemma~\ref{lem:separation} and $T_n \to \infty$ is to account for the intersection with $\XX$.  The exponential term outside the integral results from the bound in Lemma~\ref{lem:separation} and the fact that $n \eps_n^2 = 1$.  On the event $\mathscr{B}_n^c$, the exponent in the integrand is bounded by $\omega b_n n^{1/2} \|\theta-\theta^\star\|_2$, so 
\begin{align*}
N_n(A_n \cap \XX) \, 1(\mathscr{B}_n^c) & \leq \sum_{t=1}^{T_n} e^{-C\omega t^2 a_n^2} e^{\omega (t+1) a_n b_n} \, \Pi(\{\theta: \|\theta-\theta^\star\|_2 < (t+1)a_n \eps_n \}) \\
& \lesssim (a_n \eps_n)^d \sum_{t=1}^\infty e^{-\omega a_n^2 t^2 (C - 2b_n/a_n )} (t+1)^d,
\end{align*}
where the last inequality uses the fact that the prior for $\theta$ has a bounded density on $\XX$.  Since $b_n \ll a_n$, the difference in the exponent will be bigger than some $k > 0$ for sufficiently large $n$.  Therefore, the above sum is $\lesssim e^{-\omega k a_n^2}$, which proves the claim. 
\end{proof}

\subsection{Proof of Theorem~\ref{thm:rate}}
\label{proof:rate}

To prove Theorem~\ref{thm:rate}, we need to combine Lemmas~\ref{lem:consistency}--\ref{lem:num}.  Towards this, write 
\begin{align*}
\Pi_n(A_n) & \leq \Pi_n(A_n \cap \XX) + \Pi_n(\XX^c) \\
& = \frac{N_n(A_n \cap \XX)}{D_n} + \Pi_n(\XX^c) \\
& = \frac{N_n(A_n \cap \XX)}{D_n} \, 1(\mathscr{B}_n^c \cap \mathscr{C}_n^c) + \frac{N_n(A_n \cap \XX)}{D_n} \, 1(\mathscr{B}_n \cup \mathscr{C}_n) + \Pi_n(\XX^c) \\
& \leq \frac{2 N_n(A_n \cap \XX) \, 1(\mathscr{B}_n^c)}{\eps_n^d e^{-\omega(1+c_n)}} + 1(\mathscr{B}_n \cup \mathscr{C}_n) + \Pi_n(\XX^c). 
\end{align*}
Taking expectation, we get 
\[ P^n \Pi_n(A_n) \leq \frac{2P^n N_n(A_n \cap \XX) \, 1(\mathscr{B}_n^c)}{\eps_n^d e^{-\omega(1+c_n)}} + P^n(\mathscr{B}_n) + P^n(\mathscr{C}_n) + P^n \Pi_n(\XX^c). \]
The second, third, and fourth terms in the upper bound are $o(1)$ by Lemmas~\ref{lem:donsker}, \ref{lem:den}, and \ref{lem:consistency}, respectively.  Lemma~\ref{lem:num} says the first term in the upper bound above satisfies
\[ \lesssim a_n^d e^{-\omega (k a_n^2 - c_n)}. \]
The sequence $c_n$ is arbitrary, so we are free to take $c_n = o(a_n^2)$, in which case, the quantity in the above display vanishes, proving the claim.

\subsection{Proof of Theorem~\ref{thm:bvm}}
\label{proof:bvm}

The proof begins by showing that $e^{-\omega nR_n(\theta)}$ satisfies a locally asymptotic normality condition, that is, for every compact set $K \subset \RR^d$
\begin{equation}
\sup_{h \in K}    \bigl| \log s_n(h)-\omega h^\top V_{\theta^{\star}}\Delta_{n,\theta^{\star}}-\tfrac{\omega}{2}h^\top V_{\theta^{\star}}h \bigr| \to 0, \quad \text{in $P^n$-probability}.
\label{eq17}
\end{equation}
where $s_n(h)= e^{-\omega n \{R_n(\theta^{\star}+h n^{-1/2}) - R_n(\theta^{\star})\}}$. To show this, 
\begin{align*}
    -\omega n \{ R_n(\theta^{\star}+h n^{-1/2}) - R_n(\theta^{\star})\} & = -\omega n\{ \PP_n(\ell_{\theta^{\star}+hn^{-1/2}}-\ell_{\theta^{\star}}) \\
    & =-\omega nP(\ell_{\theta^{\star}+hn^{-1/2}}-\ell_{\theta^{\star}})-\omega n^{1/2} \GG_n(\ell_{\theta^{\star}+hn^{-1/2}}-\ell_{\theta^{\star}}).
\end{align*}
Since the loss $\ell_\theta$ is Lipschitz, it follows from Lemma~19.31 in \citet{van2000asymptotic}, that 
\begin{equation}
    \GG_n\{ \omega n^{1/2} (\ell_{\theta^{\star}+h n^{-1/2}}-\ell_{\theta^{\star}})-\omega h^\top \dot{\ell}_{\theta^{\star}} \} \to 0, \quad \text{in $P^n$-probability}.
\end{equation}
Since $R$ is twice differentiable at $\theta^{\star}$ with second derivative matrix $V_{\theta^{\star}}$, \eqref{eq17} holds.  Given that empirical risk difference has a suitable locally quadratic representation, it is intuitively clear that the Gibbs posterior distribution will take on a Gaussian shape.  Confirming this intuition requires some care, but one can follow exactly the arguments used to prove Theorem~2.1 in \citet{kleijn2012bernstein}.  Indeed, their setup concerns a sequence of statistical models with density functions $p_\theta^{(n)}$, and their misspecified Bayesian posterior distribution corresponds to our Gibbs posterior with the empirical risk replaced by $R_n(\theta) = -n^{-1}\log p_\theta^{(n)}$ and, of course, $\omega=1$.  They show that if the $\omega=1$ version \eqref{eq17} holds for their choice of $R_n$, and if the misspecified Bayes posterior concentrates around $\theta^\star$ at the root-$n$ rate, then the Bernstein--von Mises result holds.  Our Gibbs formulation uses the empirical risk $R_n$, we have established \eqref{eq17} and the Gibbs posterior's root-$n$ rate from Theorem~\ref{thm:rate}, so checking the Bernstein--von Mises property follows exactly the same steps as in Kleijn and van der Vaart.

\ifthenelse{1=1}{}{
Next, we prove the assertion of Theorem 2 conditional on a compact set $K \subset \RR^d$. Define random functions $f_n: K \times K \mapsto \RR$,
\[ f_n(g,h)=\bigg(1-\frac{\phi_n(h)s_n(g)\pi_n(g)}{\phi_n(g)s_n(h)\pi_n(h)}\bigg)_{+}, \]
where $\phi_n$ is the $\nm_d(\omega\Delta_{n,\theta^{\star}},{(\omega V_{\theta^{\star}})}^{-1})$ density, $\pi_n$ is the prior density of the rescaled parameter $h=n^{1/2}(\theta-\theta^{\star})$, $s_n$ is defined as 
\[ s_n(h)= e^{-\omega n \{R_n(\theta^{\star}+h n^{-1/2}) - R_n(\theta^{\star})\}}, \]
and the subscript ``+'' corresponds to the positive part, i.e., $\psi(x)_+ = 0 \vee \psi(x)$.  For any two sequences $g_n,h_n \in K$, we have $\pi_n(g_n)/\pi_n(h_n) \to 1$ and, hence, 
\begin{align*}
    2\log&\frac{\phi_n(h_n)s_n(g_n)\pi_n(g_n)}{\phi_n(g_n)s_n(h_n)\pi_n(h_n)}\\
    &=2(g_n-h_n)^\top \omega V_{\theta^{\star}}\Delta_{n,\theta^{\star}}+\omega h_n^ \top V_{\theta^{\star}}h_n- \omega g_n^\top V_{\theta^{\star}}g_n+o_P(1)\\
    & \qquad -(g_n-\omega\Delta_{n,\theta^{\star}})^\top \omega V_{\theta^{\star}}(g_n-\omega\Delta_{n,\theta^{\star}})+ (h_n-\omega\Delta_{n,\theta^{\star}})^\top \omega V_{\theta^{\star}}(h_n-\omega\Delta_{n,\theta^{\star}})\\
    &=o_P(1),
\end{align*}
as $n \to \infty$. Since $f_n$ depends continuously on $(g,\ h)$ and $K \times K$ is compact, we can conclude that
\[ \sup_{g,h \in K} f_n(g,h) \to 0 \quad \text{as $n \to \infty$, in $P^n$-probability}. \]
Assume $K$ contains a neighborhood of 0, so that $\Phi_n(K)>0$, where $\Phi_n$ denotes the $\nm_d(\Delta_{n,\theta^{\star}},(\omega V_{\theta^{\star}})^{-1})$ probability measure, which depends on data through $\Delta_{n,\theta^\star}$. Let $C_n$ be the event that $\Pi_n(K)>0$. For given $\eta>0$
\begin{equation*}
    \Omega_n=\{\sup_{g,h \in K}f_n(g,h)\leq \eta\}.
\end{equation*}
Also 
\begin{align}
    P^{(n)}\Vert \Pi_n^{K}-\Phi_n^K \Vert \one_{C_n} =& P^{(n)}\Vert \Pi_n^K-\Phi_n^K\Vert \one_{\Omega_n \cap C_n}+P^{(n)}\Vert \Pi_n^K-\Phi_n^K\Vert \one_{\Omega_n^c \cap C_n}\\
    \leq & P^{(n)}\Vert \Pi_n^K-\Phi_n^K\Vert \one_{\Omega_n \cap C_n}+2P^{(n)}(\Omega_n^c \cap C_n),
    \label{eq20}
\end{align}
since the total variation norm is bounded by 2. The 2nd term is $o(1)$ from \eqref{eq17}. The first term can be calculated as
\begin{align*}
    \frac{1}{2}P^{(n)}\Vert \Pi_n^K-&\Phi_n^k \Vert \one_{\Omega_n \cap C_n}=P^{(n)}\int \big(1-\frac{\mathrm{d}\Phi_n^K}{\mathrm{d}\Pi_n^K}\big)_{+}\mathrm{d}\Pi_n^K\one_{\Omega_n \cap C_n}\\
    &=P^{(n)}\int_K\bigg(1-\int_K\frac{s_n(g)\pi_n(g)\phi_n^K(h)}{s_n(h)\pi_n(h)\phi_n^K(g)}\mathrm{d}\Phi_n^K(g)\bigg)_+\mathrm{d}\Pi_n^K(h)\one_{\Omega_n \cap C_n}.
\end{align*}
For every $g,h \in K$, $\phi_n^K(h)/\phi_n^K(g)=\phi_n(g)/\phi_n(h)$, since on $K$, $\phi_n^K$ differs from $\phi_n$ only by a normalization factor. \par
By Jensen's inequality $(1-E(Y))^+ \leq E(1-Y)^+$. Hence
\begin{align*}
  \frac{1}{2}P^{(n)}\Vert \Pi_n^K-&\Phi_n^k \Vert \one_{\Omega_n \cap C_n}\\
  &\leq P^{(n)}\int \bigg(1-\frac{s_n(g)\pi_n(g)\phi_n^K(h)}{s_n(h)\pi_n(h)\phi_n^K(g)}\bigg)_+\mathrm{d}\Phi_n^K(g)\mathrm{d}\Pi_n^K(h)\one_{\Omega_n \cap C_n}\\
  & \leq P^{(n)}\int \sup_{g,h \in K}f_n(g,h)\one_{\Omega_n \cap C_n}\mathrm{d}\Phi_n^K(g)\mathrm{d}\Pi_n^K(h)\leq \eta .
\end{align*}
By Dominated Convergence Theorem and Theorem 1, we can conclude that for every compact set $K \subset \mathbb{R}^k$ containing a neighborhood of $0$, $P^{(n)}\Vert \Pi_n^K-\Phi_n^K\Vert \one_{C_n} \rightarrow 0$.\par
Now let $K_m$ be a sequence of balls centered at $0$ with radii $M_m \rightarrow \infty$. For each $m\geq1$, the above assertion holds. If we take a sequence of balls $K_n$ that traverses the sequence $K_m$ slowly enough, convergence to zero can still be guaranteed. Also the events $C_n=\{\Pi_n(K_n)>0\}$ satisfy $P^{(n)}(C_n)\rightarrow 1$ by Theorem 1. Therefore we can conclude that there exists a sequence of radii $M_n \rightarrow \infty$ and $P^{(n)}\Vert \Pi_n^{K_n}-\Phi_n^{K_n}\Vert. \rightarrow 0$. Now
\begin{equation}
    P^{(n)}\Vert \Pi_n-\Phi_n \Vert \leq P^{(n)}\Vert \Pi_n-\Pi_n^{K_n} \Vert + P^{(n)}\Vert \Pi_n^{K_n}-\Phi_n^{K_n}\Vert + P^{(n)}\Vert \Phi_n^{K_n}-\Phi_n \Vert.
\end{equation}
Here $\Vert \Pi_n-\Pi_n^{K_n}\Vert$ is bounded above by $2\Pi_n(K_n^c)$. From Theorem 1, the first term in the RHS of (12) converges to zero.  For the third term, the central limit theorem implies that $\Delta_{n,\theta^{\star}}$ converges in distribution to a $k$-dimensional normal distribution with mean vector $0$ and covariance matrix $\omega^2V_{\theta^{\star}}^{-1}P(\dot{\ell}_{\theta^{\star}}\dot{\ell}_{\theta^{\star}}^T)V_{\theta^{\star}}^{-1}$. Hence $\Delta_{n,\theta^{\star}}$ is uniformly tight.  Then it follows from Lemma~5.2 in \citet{kleijn2012bernstein} that $P^{(n)}\Vert \Phi_n^{K_n}-\Phi_n \Vert \leq 2P^{(n)}\Phi_n(\mathbb{R}^k\setminus K_n) \rightarrow 0$.
}

\ifthenelse{1=1}{
\bibliographystyle{apalike}
\bibliography{sample.bib}
}{

}

\end{document}